\documentclass[11pt,reqno]{article}
\usepackage{amsthm,amsmath, amssymb, amsopn, amsfonts,enumitem}
\usepackage{cases}
\usepackage[margin=1in]{geometry}
\usepackage{subcaption}
\usepackage[colorlinks,citecolor=blue,urlcolor=blue]{hyperref}
\usepackage{graphicx}
\usepackage{tikz}
\usetikzlibrary{decorations.fractals}
\usetikzlibrary{patterns,arrows,shapes,decorations.pathreplacing}
\usepackage{pgfplots}
\usepackage{stmaryrd} 
\usepackage{soul,xcolor}
\setstcolor{red}
\usepackage{xcolor}

\numberwithin{equation}{section}
\theoremstyle{plain}
\newtheorem{Definition}{Definition}[section]
\newtheorem{Remark}{Remark}[section]
\newtheorem{Theorem}{Theorem}[section]
\newtheorem{Lemma}{Lemma}[section]
\newtheorem{Proposition}{Proposition}[section]
\newtheorem{Corollary}{Corollary}[section]
\newtheorem{Assumption}{Assumption}[section]

\newcommand{\be}{\begin{equation}}
\newcommand{\ee}{\end{equation}}
\newcommand{\bee}{\begin{equation*}}
\newcommand{\eee}{\end{equation*}}
\newcommand{\bi}{\begin{itemize}}
\newcommand{\ei}{\end{itemize}}
\def \E{\mathbb{E}}
\def \F{\mathbb{F}}
\def \N{\mathbb{N}}
\def \Z{\mathbb{Z}}
\def \P{\mathbb{P}}

\def \R{\mathbb{R}}
\def \X{\mathbb{X}}

\def \Cc{{\mathcal C}}

\def \Ec{{\mathcal E}}
\def \Fc{{\mathcal F}}

\def \eps{\varepsilon}

\def \rhotilde{\tilde{\rho}}

\def \Jtilde{\tilde{J}}

\def \xhstarv{\left\lfloor\frac{x^*}{\sqrt{h}} \right\rfloor}
\def \xhstarh{\lfloor x^*/\sqrt{h} \rfloor}

\makeatletter
\newcommand{\setword}[2]{%
	\phantomsection
	#1\def\@currentlabel{\unexpanded{#1}}\label{#2}%
}
\makeatother

\title{Binomial-tree approximation for time-inconsistent stopping}

\author{Erhan Bayraktar\thanks{
		Department of Mathematics, University of Michigan, Ann Arbor, email: \texttt{erhan@umich.edu}. E. Bayraktar is partially supported by the National Science Foundation under grant DMS2106556 and by the Susan M. Smith chair.}	
	\and Zhenhua Wang\thanks{
		 Zhongtai Securities Institute for Financial Studies and School of Mathematics, Shandong University, email: \texttt{zhenhuaw@sdu.edu.cn}. }
	\and Zhou Zhou\thanks{School of Mathematics and Statistics, University of Sydney, Australia, email:
		\texttt{zhou.zhou@sydney.edu.au}.}
}

\begin{document}

\maketitle

\date{}

\begin{abstract}
For time-inconsistent stopping in a one-dimensional diffusion setup, we investigate how to use discrete-time models to approximate the original problem. In particular, we consider the value function $V(\cdot)$ induced by all mild equilibria in the continuous-time problem, as well as the value $V^h(\cdot)$ associated with the equilibria in a binomial-tree setting with time step size $h$. We show that $\lim_{h\to 0+}V^h\leq V$. We provide an example showing that the exact convergence may fail. Then we relax the set of equilibria and consider the value $V_\eps^h(\cdot)$ induced by $\eps$-equilibria in the binomial-tree model. We prove that $\lim_{\eps\to 0+}\lim_{h\to 0+}V_\eps^h=V$.
\end{abstract}

{\bf Keywords:} Time-inconsistent stopping, binomial-tree approximation, $\eps$-equilibrium, weighted discount function.\\

\section{Introduction}

Time inconsistency refers to a phenomenon where a strategy planned to be optimal today may no longer be optimal from a future's perspective due to the change of preferences. A common approach to address this time inconsistency is to use a game-theoretic framework and look for an equilibrium strategy: given future selves use this strategy, the current self has no incentive to deviate. For equilibrium strategies of time-inconsistent control problems, we refer to \cite{MR4387700, MR4328502, MR4288523} and the references therein.

Very recently, there has been a lot of research on equilibrium strategies for time-inconsistent stopping. See \cite{bwz-2023, MR4205889, bodnariu2024local,MR3880244, MR4080735,huang2018time,MR4273542,MR3911711,MR4332966,ssrn4431616}, to name a few. It is worth noting that most of these works focus on the characterization and/or construction of equilibria. Apart from these works, \cite{MR4609255,MR4205889,ssrn4431616} compare different notions of equilibria in continuous time under non-exponential discounting; in particular, \cite{MR4609255,MR4205889} show that an optimal equilibrium (its existence is proved in \cite{MR4116459}), is also a weak and strong equilibrium under certain assumptions. \cite{bwz-2022,bwz-2023} investigate the stability of equilibrium-associated value function under a perturbation of the payoff function and the transition law of the underlying process. In a mean-variance setup, \cite{MR4276004} analyzes another kind of stability regarding whether a strategy near an equilibrium would converges to that equilibrium under policy adjustment.

The focus of this work is very different from the previous literature on time-inconsistent stopping. In this paper, we consider a time-inconsistent stopping problem under one-dimensional diffusion with a weighted discount function in infinite horizon. We are interested in how to approximate the time-inconsistent problem by a discrete-time discrete-space model. To be more specific, we consider the value function $V(\cdot)$ induced by all mild equilibria (equivalently by an optimal equilibrium, see Definition \ref{def.mild.optimal}) in continuous time, as well as the value $V^h(\cdot)$ associated with the equilibria in a carefully-designed binomial tree model with time-step size $h$. As our first main result, we show that $\limsup_{h\to 0+}V^h\leq V$. We provide an example (see Section \ref{sec:4}) showing that a strict inequality is possible. Next, we relax the equilibria set in the binomial tree model and consider the value $V_\eps^h$ induced by all $\eps$-equilibria. As our second main result, we prove that $\lim_{\eps\to 0+}\lim_{h\to 0+}V_\eps^h=V$.

A key step to establish the two main results is to show the convergence of the expected stopping value from discrete to continuous time, uniform in the starting position and stopping region. Thanks to the form of a weighted discount function, we are able to express the expected stopping value by an integral   where the integrand is the stopping value associated with some exponential discounting. Then we use a PDE approach and finite difference method to show the uniform convergence with respect to (w.r.t.) each exponential discounting, and thus the uniform convergence w.r.t. the weighted discounting under certain integrability assumption.

Our work makes a very novel contribution in the literature of time-inconsistent control and stopping. To the best of our knowledge, our paper is the first to consider the discrete-time approximation for continuous-time time-inconsistent stopping problems. Let us mention the works on continuous-time time-consistent control where a time discretization is involved, such as \cite{MR3738841,MR2822686}. In these papers, the time discretization is used to construct an approximate continuous-time equilibrium, and thus the focus of these papers is quite different from ours.

Our paper demonstrates a distinct feature regarding the value functions between time-consistent and time-inconsistent case. For the time-consistent situation, it is well known that under natural assumptions, the related binomial-tree model (and some other discretized model) provides a good approximate for the continuous-time problem regarding the value function, meaning that $V^h$ converges to $V$ as the step size $h$ approaches to zero. However, that is no longer the case under time inconsistency as suggested by our first main result. Specifically, the upper limit of $V^h$ can be strictly less than $V$, as detailed in Section \ref{sec:4}.  To achieve a good approximation, it is necessary to expand the equilibrium set and consider $\eps$-equilibria in the discretized model, which is indicated by our second main result. From this point of view, we believe our results would potentially be very useful for numerical computation for the value of time-inconsistent stopping problems.

{\color{black}We also want to emphasize that our results align with the observation in game theory. That is, when perturbing the underlying model a little bit (we think discretization as one kind of perturbation, meaning that the discretized process for small step size $h$ can be treated as a perturbation of the original continuous-time process), the associated equilibria and values may change significantly;\footnote{More specifically, in a game when the model (e.g., payoff functions or dynamics of state processes) changes a little bit, a player's best response (with respect to other players' strategies) may change a lot. For this reason, a small perturbation of the game model could possibly lead to a significant change of equilibria and associated values.} to ``stablize'' this change caused by the perturbation of the model, one needs to consider $\eps$-equilibria instead; see e.g., \cite{feinstein2022continuity,feinstein2022dynamic}.}

The rest of the paper is organized as follows. In the next section, we provide the continuous-time framework, the associated binomial tree model, as well as the main results of this paper. In Section \ref{sec:3}, we give the proof of our main results. In Section \ref{sec:4} an example showing that the exact convergence may fail for the value function when we only consider the set of exact equilibria in the discretized model.
{\color{black} Section \ref{sec:5} provides a case study of the stopping of an American put option.}

\section{Setup and main results}\label{sec:preliminaries}

In this section, we formulate the continuous-time problem and the corresponding discrete-time model. We provide the main results of this paper at the end of this section.

\subsection{Continuous-time setup}
Denote $\N:=\{0,1,2,\dotso\}$ and $\N_+:= \{1,2,\dotso\}$. 
{\color{black}
For an open set $E\subset \R$, we denote the following H\"older norm on $E$ as
	$$
	\|f\|_{\Cc^k(E)}:=\sum_{i=0}^k\sup_{x\in E}|f^{(i)}(x)|\quad \text{for } k\in \N,$$
provided that the right-hand side above exists.} 
Let $(\Omega, \P, (\Fc_t)_{t}, \F)$ be a filtered probability space supporting a 1-dimensional Brownian motion $W=(W_t)_{t\geq 0}$. Consider a 1-dimensional diffusion $X$ given by 
\begin{equation}\label{e0}
	dX_t=\mu(X_t)dt+\sigma(X_t)dW_t,
\end{equation}
taking values in $\X:=\R$ for any $X_0=x\in\X$. Here $\mu,\sigma: \X\to \R$ are some Borel-measurable functions. {\color{black} Later $\mu, \sigma$ are assumed to be bounded and Lipschitz over $\X$, making \eqref{e0} admitting a unique strong solution.} Let $\F^X:=(\Fc_t^X)$ be the filtration generated by $X$, and $\mathcal{T}$ be the set of $\F^X$-stopping times. {\color{black}Denote $\E_x[\cdot]:=\E[\cdot|X_0=x]$.} Consider the stopping problem
 
\begin{equation}\label{e1}
\sup_{\tau\in\mathcal{T}}\E_x[\delta(\tau)f(X_\tau)]\quad \text{for}\; x\in \R,
\end{equation}
where  $f:\X\to[0,\infty)$ is Borel measurable, $\delta:[0,\infty)\to[0,1]$ is a weighted discount function of the form
\begin{equation}\label{e3}
\delta(t)=\int_0^\infty e^{-rt}dF(r),\quad\forall\,t\in [0,\infty).
\end{equation}
Here $F(r): [0,\infty) \mapsto [0, 1]$ is a cumulative distribution function with $\int_0^\infty rdF(r)<\infty$. Assume $\lim_{t\to\infty}\delta(t)=0$.

\begin{Remark}
Many commonly used non-exponential discount functions can be written in the form \eqref{e3}, such as the hyperbolic, generalized hyperbolic, and pseudo exponential discounting. We refer to \cite{MR4124420} for a detailed discussion about weighted discount function. Moreover, \cite[Proposition 1]{MR4124420} indicates that weighted discount functions satisfy the following property:
	\be\label{eq.delta.logsub}
	\delta(t+s)\geq \delta(t)\delta(s),\quad \forall\, t,s\geq 0.
	\ee 
\end{Remark}

Set $\delta(\infty)f(x):=0$ for any $x\in\X$. We make the following assumptions.
  
 \begin{Assumption}\label{assum:1}
 We assume the following.
 \begin{itemize}
 \item[(i)] $f$ is bounded and Lipschitz continuous.
\item [(ii)] $\|\mu\|_{\Cc^2(\X)}, \|\sigma\|_{\Cc^2(\X)}$ are finite and $\inf_{x\in \X}\sigma(x)>0$. 
 \end{itemize} 
 \end{Assumption}

Given a closed set $S\subset \X$, denote
\be  
\rho_S:= \inf\{t>0: X_t\in S\}\quad\text{and}\quad J(x, S):= \E_x \left[ \delta(\rho_S) f(X_{\rho_S})\right].
\ee 
As $\delta(\cdot)$ is non-exponential, the problem \eqref{e1} may be time-inconsistent. Following \cite{MR4205889}, we define mild equilibria and optimal mild equilibria as follows.

\begin{Definition}[Mild equilibria and optimal mild equilibria]\label{def.mild.optimal}
	A closed set $S\subset\X$ is said to be  a mild equilibrium (in continuous time), if
{\color{black}
\begin{numcases}{}
	\label{eq:def.mild1} f(x)\leq J(x,S),\quad \forall x\notin S,	\\
	\label{eq:def.mild0}f(x)\geq J(x,S),\quad \forall x\in S.
\end{numcases}
}
	Denote $\Ec$ the set of mild equilibria.
	A mild equilibrium $S$ is  said to be optimal, if for any other mild equilibrium $R\in\Ec$, 
	$$J(x,S)\geq J(x,R),\quad \forall\,x\in\X.$$
\end{Definition}

\begin{Remark}\label{rm:def.equilibrium}
{\color{black} Note that $f(x)$ represents the value for immediate stopping, while $J(x,S)$ represents the value for continuing until hitting the stopping region $S$. Thus, \eqref{eq:def.mild1} means that, when $x\notin S$, it is better to continue than to stop, indicating no incentive to deviate from continuing. The same logic applies to the case $x\in S$ in \eqref{eq:def.mild0}, i.e., no incentive to switch from stopping to continuing. By Assumption \ref{assum:1} (ii) and the one-dimensional setting, $\rho_S=0$ a.s. for $x\in S$, and thus $J(x,S)=f(x)$ for $x\in S$, making \eqref{eq:def.mild0} hold trivially.  

Assumption \ref{assum:1} (ii) and the one-dimensional setting further imply that $\rho_S=\rho_{\overline{S}}$ a.s for all $x\in \X$ and Borel set $S\subset\X$. Therefore, $J(\cdot, S)=J(\cdot, \overline S)$, which means that any Borel stopping region is equivalent to its closure. That is why we only focus on closed sets in Definition \ref{def.mild.optimal}. We refer to \cite[Section 4.1]{MR4116459} for a more detailed explanation. 

Apart from mild equilibria, there are stronger notions of equilibria in the continuous-time setup. We refer to \cite{MR4205889,MR4276004} for a detailed discussion and comparison. It is shown in \cite[Theorem 4.1]{MR4116459} that there exists a mild equilibrium that is the smallest in terms of inclusion and optimal in the sense of pointwise dominance for value functions. We restate this result as a lemma in the following.}
\end{Remark}

\begin{Lemma}\label{lm.optimalmild}
	Suppose $\mu,\sigma,f$ are bounded and continuous, and $\sigma(x)>0$ for $x\in\X$. Then 
	$$S_*:= \cap_{S\in \Ec} S$$
	is an optimal mild equilibrium.
\end{Lemma}
We define the value induced by mild equilibria in continuous time
$$V(x):= \sup_{S\in\Ec}J(x,S)=J(x,S_*).$$
Our goal in this paper is to approximate $V$ using ($\eps$-)equilibrium strategies in discrete time. Let us introduce the time and state space discretization in the next subsection.

\subsection{Time and state space discretization}
Given $h>0$, let $\X^h:= \{x_k^h\}_{k\in \Z},(d_{k,\pm})_{k\in \Z}\subset\R$ defined recursively as follows:
$$x_0^h=0,\quad d_{0,+}=d_{0,-}=\sigma(x_0^h)\sqrt{h},$$
\be\label{eq:discrete.x}
\begin{cases}
x_k^h=x_{k-1}^h+d_{k-1,+},\quad d_{k,-}=d_{k-1,+},\quad d_{k,+}=\frac{\sigma^2(x_k^h)h}{d_{k,-}}, &k=1,2,\dotso\\
 x_k^h=x_{k+1}^h-d_{k+1,-},\quad d_{k,+}=d_{k+1,-},\quad d_{k,-}=\frac{\sigma^2(x_k^h)h}{d_{k,+}}, &k=-1,-2,\dotso.
\end{cases}
\ee 

\begin{Assumption}\label{a3}
There exists some $c_1,c_2>0$ independent of $h$ such that
\begin{equation}\label{eq100}
c_1\sqrt{h}\leq d_{k,\pm}\leq c_2\sqrt{h}.
\end{equation}
\end{Assumption}

\begin{Remark}
Suppose $\underline\sigma:=\inf_{x\in \X}\sigma(x)>0$ and the total variation of $\sigma(\cdot)$ is finite, i.e., there exists a constant $C>0$ such that for any $(y_k)_{k\in\Z}\subset\R$ with $y_k<y_{k+1}$ for $k\in\Z$, it holds that
$$\sum_{k\in\Z}|\sigma(y_{k+1})-\sigma(y_k)|\leq C.$$
Then Assumption \ref{a3} is satisfied. Indeed, when $k$ is positive and even, we can compute that
$$d_{k,+}=\frac{\Pi_{i=1}^{k/2}\sigma^2(x_{2i}^h)}{\Pi_{i=1}^{k/2}\sigma^2(x_{2i-1}^h)}\sigma(0)\sqrt{h}.$$
We have that
\begin{align*}
\left|\ln\left(\frac{\Pi_{i=1}^{k/2}\sigma^2(x_{2i}^h)}{\Pi_{i=1}^{k/2}\sigma^2(x_{2i-1}^h)}\right)\right|\leq2\sum_{i=1}^{k/2}\left|\ln\sigma(x_{2i}^h)-\ln\sigma(x_{2i-1}^h)\right|\leq\frac{2}{\underline\sigma}\sum_{i=1}^{k/2}|\sigma(x_{2i}^h)-\sigma(x_{2i-1}^h)|\leq \frac{2C}{\underline\sigma}.
\end{align*}
As a result, we can take $c_1:=\sigma(0)e^{-2C/{\underline\sigma}}$ and $c_2:=\sigma(0)e^{2C/{\underline\sigma}}$ in \eqref{eq100} for $d_{k,+}$ with $k$ positive and even. The other cases for $d_{k,\pm}$ can be proved similarly.
\end{Remark}

For $x\in \X$, define $x(h):=x^h_k$ when $x\in [x^h_k, x^h_{k+1})$. For any closed set $S\subset \X$, denote by 
$$S(h):=\left\{ x_k^h:\ [x^h_k, x^h_{k+1})\cap S \ne \emptyset,\ k\in\Z \right\},$$
Assumption \ref{a3} ensures that 
\be\label{eq:S.xunf} 
\sup_{S \text{ closed }} dist(S(h), S)\to 0, \quad \sup_{x\in \X} |x(h)-x|\to 0, \quad \text{ as } h\to 0+.
\ee 
Denote by $X^h:=(X^h_{n})_{n\in \N}$ the discrete-time binomial tree on $\X^h$ with time-step size $\Delta t = h$ and transition probabilities equal to
\be\label{eq:discrete.P}  
\P\left(X^h_{n+1}= x^h_{k\pm1} \mid X^h_{n}= x^h_{k}\right)=
 \frac{d_{k,\mp}}{d_{k,+}+d_{k,-}}\pm\frac{\mu(x^h_k)h}{d_{k,+}+d_{k,-}}=: p^h_{k,\pm}.
 \ee 
Let
$$\delta^h(k):= \int_0^\infty (1+r h)^{-k} dF(r),\quad k\in\N.$$
It is easy to see that
$$
\lim_{h\to 0+}\delta^{h}(\lfloor t/h\rfloor)=\delta(t),\quad t\geq 0.
$$
For a closed set $S\subset \X$ and $x\in\X$, denote
\begin{align*}  
\rho_S(h):= \inf\{k\in\N_+:\ X^h_k\in S(h)\}\quad\text{and}\quad J^h(x, S)=\E_{x(h)}^h\left[\delta^h(\rho_S(h)) f(X^h_{\rho_S(h)}) \right],\\
\tilde\rho_S(h):= \inf\{k\in\N:\ X^h_k\in S(h)\}\quad\text{and}\quad \tilde J^h(x, S)=\E_{x(h)}^h\left[\delta^h(\tilde\rho_S(h)) f(X^h_{\tilde\rho_S(h)}) \right].
\end{align*}
where $\E_y^h[\cdot]=\E^h[\cdot|X_0^h=y]$ for $y\in\X^h$, and we write $\E^h$ to emphasis that the expectation is under the binomial-tree setting. 
Following \cite{bwz-2023}, we define ($\eps$-)equilibria in the binomial-tree model as follows.
\begin{Definition}\label{d2}
Let $\eps\geq 0$ and $S\subset\X^h$. $S$ is called an $\varepsilon$-equilibrium (under the binomial-tree setting), if 
	\be \label{eq.def.epsequi}
\begin{cases}
	f(x)\leq J^h(x,S)+\eps,&x\notin S,\\
	f(x)+\eps\geq J^h(x,S),& x\in S.
\end{cases}
	\ee 
Denote $\Ec^h_\eps$ the set of $\varepsilon$-equilibria. $S$ is called an equilibrium if the above inequalities hold with $\eps=0$. We also denote $\Ec^h:=\Ec^h_0$. $S$ is called an optimal equilibrium, if for any equilibrium $R\in\Ec^h$, it holds that
\begin{equation}\label{e2}
\tilde J^h(x,S)\geq \tilde J^h(x,R)\quad(\Longleftrightarrow  J^h(x,S)\geq J^h(x,R)),\quad\forall\,x\in\X^h.
\end{equation}
\end{Definition}

\begin{Remark}
$\tilde J^h(x,R)=J^h(x,R)\vee f(x)$ can be thought of as the value associated with the equilibrium $R$.
{\color{black} Note in continuous time, we have $J(x,S)=J(x,S)\vee f(x)$ for any mild equilibrium $S$, thanks to Assumption \ref{assum:1}.}
\end{Remark}

The following result is from \cite[Lemma 2.1]{bwz-2023}.
\begin{Lemma}
Assume $f$ is bounded. Then
$$S_*^h:=\cap_{S\in\Ec^h} S$$
is an optimal equilibria {\color{black}(under the Binomial-tree setting)}.
\end{Lemma}

Denote the value induced by all ($\eps$-)equilibria in the binomial-tree setup as follows
\begin{equation}\label{e6}
V_\eps^h(x):=\sup_{S\in\Ec^h_\eps}\tilde J^h(x,S)\quad\text{and}\quad V^h(x):=\sup_{S\in\Ec^h}\tilde J^h(x,S)=\tilde J^h(x, S_*^h),\quad x\in\X^h.
\end{equation}

\subsection{Main results}
We are ready to state the main results of this paper. Their proofs are provided in the next section.
\begin{Theorem}\label{thm:semicts}
	Let Assumptions \ref{assum:1} and \ref{a3} hold. Then
\begin{equation}\label{eq001}
	\limsup_{h\to 0+} V^h(x) \leq V(x),\quad x\in \X.
\end{equation}
\end{Theorem}

\begin{Remark}
The exact equality may fail in general. We provide such an example in Section \ref{sec:3}. To recover the exact equality, we need to relax the equilibrium set and consider the value induced by $\eps$-equilibria in the binomial-tree case, as stated in the following result.
\end{Remark}

\begin{Theorem}\label{thm:epscts}
Let Assumptions \ref{assum:1} and \ref{a3} hold. Then
$$
\lim_{\eps\to 0+} \left(\limsup_{h\to 0+} V^h_\eps(x) \right)=\lim_{\eps\to 0+}\left( \liminf_{h\to 0+} V^h_\eps(x) \right)= V(x),\quad \forall x\in \X.
$$
\end{Theorem}

\section{Proofs of Theorems \ref{thm:semicts} and \ref{thm:epscts}}\label{sec:3}

For $r\in (0,\infty)$, define
$$J(x, S; r):= \E_x \left[ e^{-r \rho_S} f(X_{\rho_S})\right].$$
Fubini's Theorem gives that
	\begin{align}\label{eq:J.integral}
J(\cdot ,S)=\int_0^\infty J(\cdot, S; r) dF(r).
	\end{align}
The following assumption will be used for some results in the rest of this section.
\begin{Assumption}\label{a2}
$\mu,\sigma,f$ are bounded and Lipschitz continuous, and $\inf_{x\in\X}\sigma(x)>0$.
\end{Assumption}

\begin{Lemma}\label{lm:J.r}
The following statements hold with $C_0$ being a finite constant independent of $r$.
\bi 
\item[(a)] If Assumption \ref{a2} holds, then
$$
\sup_{S \text{ closed}} \|J(x, S; r)\|_{\Cc^2(\X\setminus S)} \leq C_0 (1+r),\quad \forall r\in(0,\infty).
$$
\item[(b)] If Assumption \ref{assum:1} holds, then
$$
\sup_{S \text{ closed }}\|J(x, S; r)\|_{\Cc^4(\X\setminus S)}\leq  C_0(1+r^2),\quad \forall r\in(0,\infty).
$$
\ei 
\end{Lemma}

\begin{proof}
(a) For any closed set $S$ and $r>0$, $J(x,S; r)$ solves the {\color{black}following PDE}
\be\label{eq:lmunif.1} 
\begin{cases}
-r v(x)+\mu(x)v_x(x)+\frac{1}{2} \sigma^2(x) v_{xx}(x)=0 \quad \forall x\in \X\setminus S,\\
v(x)=f(x)\quad x\in \partial S.
\end{cases}
\ee 
Then an argument similar to that in \cite[proof of Lemma 3.6, Step 2]{bwz-2022} gives part (a).

(b) Differentiating the elliptic equation in \eqref{eq:lmunif.1} gives that 
\be\label{eq:lmunif.3} 
\frac{1}{2} \sigma^2(x) v^{(3)}(x)+\sigma(x)\sigma_x(x)v^{(2)}(x)+\mu_x(x) v_x +\mu v^{(2)}-r v_x = 0.
\ee 
This together with part (a) and Assumption \ref{assum:1} implies that 
$$
\sup_{S \text{ closed }}\|v^{(3)}\|_{{\color{black}\Cc^0(\X\setminus S)}}\leq C_0 (1+r^2).
$$
The a similar argument by differentiating \eqref{eq:lmunif.3} gives part (b).
\end{proof}

\begin{Corollary}\label{lm:uniflip}
	Let Assumption \ref{assum:1} hold. Then there exists constant $L>0$ such that 
	$$
	\sup_{S \text{\,closed},\, x\neq y} \frac{|J(x, S)-J(y, S)|}{|x-y|} \leq L.
	$$
\end{Corollary}
\begin{proof}
 By Lemma \ref{lm:J.r}(a), \eqref{eq:J.integral} and the assumption $\int_0^\infty rdF(r)<\infty$, we have that
$$
\sup_{S \text{ closed }}\|J(\cdot, S)\|_{\Cc^2(\X\setminus S)}\leq \int_0^\infty \left(\sup_{S \text{ closed }}\|J(\cdot, S; r)\|_{\Cc^2(\X\setminus S)} \right) dF(r)\leq C_0\int_0^\infty(1+r)dF(r)<\infty. 
$$
This together with the Lipschitz continuity of $f$ implies the result.
\end{proof}

\begin{Lemma}\label{lm:J.sqrth}
Let Assumption \ref{a2} hold. Then
\bee 
\sup_{S\subset\X \text{ closed}} \| J(\cdot, S(h); r) -J(\cdot, S; r) \|_\infty\leq C(1+r) \sqrt{h},\quad \forall h>0,
\eee 
where $C$ is a finite constant independent of $h$.
\end{Lemma}

\begin{proof}
Let $h,r>0$ and $S\subset\X$ be a closed set. For $x\in\X$, we have that
\begin{align*}
\| J(\cdot, S(h); r) -J(\cdot, S; r) \|=&\bigg|\E_x\bigg[\delta(\rho_{S(h)})1_{\{\rho_{S(h)}<\rho_S\}}\left(f(X_{\rho_{S(h)}})-J(X_{\rho_{S(h)}},S;r)\right)\\
&+\delta(\rho_{S})1_{\{\rho_{S(h)}>\rho_S\}}\left(J(X_{\rho_{S}},S(h);r)-f(X_{\rho_{S}})\right)\bigg]\bigg|
\end{align*}
Note that for any $y\in S$ there exists $y'\in S(h)$ such that $|y-y'|\leq O(\sqrt{h})$ and vice versa. Then by Lemma \ref{lm:J.r}(a) and the Lipschitz continuity of $f$, there exists $C$ independent of $h,r,S,x$ such that
$$\left|f \left( X_{\rho_{S(h)}} \right)-J \left( X_{\rho_{S(h)}},S;r \right)\right|+\left|J \left( X_{\rho_{S}},S(h);r \right) - f \left( X_{\rho_{S}} \right)\right|\leq C(1+r)\sqrt{h}.$$
The result follows.
\end{proof}

For a closed set $S\subset \X$, $r\in (0,\infty)$ and $x\in\X$, define
\be  
\begin{aligned}
	J^h(x, S; r)=\E^h_{x(h)}\left[ (1+rh)^{-\rho_S(h)} f(X^h_{\rho_S(h)}) \right]\quad \text{and}\quad \Jtilde^h(x, S; r)=\E_{x(h)}\left[ (1+rh)^{-\rhotilde_S(h)} f(X^h_{\rhotilde_S(h)}) \right].
\end{aligned}
\ee 
{\color{black}By Fubini's Theorem again, we have
$$
J^h(x, S)=\int_0^\infty J^h(x, S; r) dr\quad\text{and}\quad \tilde J^h(x, S)=\int_0^\infty \tilde J^h(x, S; r) dr.
$$}

\begin{Lemma}\label{lm:J.Jh}
Let Assumptions \ref{assum:1} and \ref{a3} hold. Then for $h> 0$ and $r>0$ we have that
\begin{align}
\sup_{S\subset\X \text{ closed}} \| \tilde J^h(\cdot, S; r)-J(\cdot, S(h); r) \|_\infty\leq C \frac{1+r^2}{r} \sqrt{h},\label{eq:lm.Jh}
\end{align}
where $C$ is a finite constant independent of $h$ and $r$.
\end{Lemma}

\begin{proof}
Throughout the proof, $C>0$ is a constant independent of $h, r$ and may vary from line to line.  
Fix  $h, r>0$. Take $x^h_k$. If $x^h_k\in S(h)$, then $\Jtilde^h(x^h_k, S; r)=f(x^h_k)=J(x^h_k, S(h); r)$. Now assume $x^h_k\notin \overline S(h)$. We apply the finite difference method in this case. Set $d_k:= d_{k,+}+d_{k,-}$.  We have that 
\begin{align*}
J_x(x^h_k, S(h); r)=& \frac{J(x^h_{k+1}, S(h); r)-J(x^h_{k-1}, S(h); r) }{d_k}+ (1+r)O(\sqrt{h}), \\
J_{xx}(x^h_k, S(h); r)=& \frac{d_{k,+}J(x^h_{k-1}, S(h); r)+d_{k,-}J(x^h_{k+1}, S(h); r)-d_kJ(x^h_{k}, S(h); r) }{ \frac{1}{2}d_kd_{k,+}d_{k,-}}+ (1+r^2)O(\sqrt{h}),
\end{align*}
where the $O(\sqrt{h})$ terms are uniform in $x_k^h$, $r$, $S$ thanks to Assumption \ref{a3} and {\color{black}Lemma \ref{lm:J.r}(b)}. Then \eqref{eq:lmunif.1} and the above estimates give that
\be\label{eq:lm.difference1}  
\begin{aligned}
&-r J(x^h_k, S(h); r)+\mu(x^h_k)\frac{J(x^h_{k+1}, S(h); r)-J(x^h_{k-1}, S(h); r) }{d_k}\\
&+{\color{black}\frac{1}{2}\sigma^2(x^h_k)}\frac{d_{k,-}J(x^h_{k+1}, S(h); r)+d_{k,+}J(x^h_{k-1}, S(h); r)-d_kJ(x^h_{k}, S(h); r) }{ \frac{1}{2}d_kd_{k,+}d_{k,-}}+(1+r^2)O(\sqrt{h})=0.
\end{aligned}
\ee 
As $x^h_k\notin S(h)$, \eqref{eq:discrete.P} tells that 
$$
\Jtilde^h(x^h_k, S; r)=(1+r h)^{-1} \left[ p^h_{k,+} \Jtilde^h(x^h_{k+1}, S; r)+ p^h_{k,-} \Jtilde^h(x^h_{k-1}, S; r) \right].
$$
This together with \eqref{eq:discrete.x} implies that 
\be\label{eq:lm.difference3}  
\begin{aligned}
	&-r \Jtilde^h(x^h_k, S; r)+\mu(x^h_k)\frac{ \Jtilde^h(x^h_{k+1}, S; r)- \Jtilde^h(x^h_{k-1}, S; r) }{d_k}\\
	&+\frac{1}{2}\sigma^2(x^h_k)\frac{ d_{k,-}\Jtilde^h(x^h_{k+1}, S; r)+ d_{k,+}\Jtilde^h(x^h_{k-1}, S; r)-d_k \Jtilde^h(x^h_{k}, S; r) }{ \frac{1}{2}d_kd_{k,+}d_{k,-}}=0.
\end{aligned}
\ee 
Then \eqref{eq:lm.difference1} minus \eqref{eq:lm.difference3} gives that 
\bee
\begin{aligned}
	-r D_k+\mu(x^h_k)\frac{D_{k+1}-D_{k-1} }{d_k}
	+\frac{1}{2}\sigma^2(x^h_k)\frac{d_{k,-}D_{k+1}+d_{k,+}D_{k-1}-d_kD_k }{ \frac{1}{2}d_kd_{k,+}d_{k,-}}+(1+r^2)O(\sqrt{h})=0,
\end{aligned}
\eee 
where $D_j:=|J(x^h_{j}, S(h); r)-\Jtilde^h(x^h_{j}, S; r)|$ for $j=k-1, k, k+1$. Then
\begin{align*}
(1+r h) D_k= &p^h_{k,+}D_{k+1}+p^h_{k,-}D_{k-1}+(1+r^2) O(h^\frac{3}{2})\\
\leq & (p^h_{k,+}+p^h_{k,-})\sup_{x^h_i\notin S(h)} |J(x^h_{i}, S; r)-\Jtilde^h(x^h_{i}, S; r)|+(1+r^2) O(h^\frac{3}{2}).
\end{align*}
Note the above holds for any $x^h_k\notin S(h)$. As a result,
$$
(1+r h) \sup_{x^h_i\notin S(h)} |J(x^h_{i}, S(h); r)-\Jtilde^h(x^h_{i}, S; r)|
\leq \sup_{x^h_i\notin S(h)} |J(x^h_{i}, S(h); r)-\Jtilde^h(x^h_{i}, S; r)|+(1+r^2) O(h^\frac{3}{2}),
$$
and thus
\be\label{eq:lm.difference5} 
\sup_{x^h_i\notin S(h)} |J(x^h_{i}, S(h); r)-\Jtilde^h(x^h_{i}, S; r)|\leq \frac{1+r^2}{r} O(\sqrt{h}).
\ee 
Notice that $\Jtilde^h(x, S; r)= \Jtilde^h(x^h_{i}, S; r)$ for all $x\in[x^h_i, x^h_{i+1})$, and $J(\cdot, S; r)$ has a uniform Lipschitz constant $C_0 (1+r)$ over all closed set $S$ due to Lemma \ref{lm:J.r}(a) and the Lipschitz continuity of $f$. Hence, \eqref{eq:lm.difference5}  implies \eqref{eq:lm.Jh}.
\end{proof}

\begin{Proposition}\label{lm:main}
	Let Assumptions \ref{assum:1} and \ref{a3} hold. Then there exists a constant $C>0$ independent of $h$ such that
	$$\sup_{S \text{ closed }} \| \tilde J^h(\cdot, S)-J(\cdot, S)\|_{\infty}\leq C \sqrt{h}.$$
\end{Proposition}

\begin{proof}
By combining Lemma \ref{lm:J.sqrth} with Lemma \ref{lm:J.Jh}, we have that
\bee 
\sup_{S \text{ closed}}\left\| \tilde J^h (\cdot, S; r) - J(\cdot, S;r)\right\|_{^\infty} \leq C\frac{1+r^2}{r} \sqrt{h},\quad \forall r>0,
\eee 
for some constant $C>0$ independent of $r,h$. Let $\eps>0$, As $\lim_{t\to\infty} \delta(t) = 0$, $F(0) = 0$. By the right-continuity of $F$, there exists $r_0 > 0$ such that $F(r_0)\leq \eps$. Then 
\begin{align*}
&\sup_{S \text{ closed}}\left\| \tilde J^h (\cdot, S) - J(\cdot, S)\right\|_{^\infty} \leq \left(\int_0^{r_0}+\int_{r_0}^\infty\right)\sup_{S \text{ is closed}}\left\| \tilde J^h (\cdot, S; r) - J(\cdot, S;r)\right\|_{^\infty}  dF(r) \\
&\leq \|f\|_{\infty} \eps+ \sqrt{h}C\int_{r_0}^\infty\frac{1+r^2}{r} dF(r)\leq \|f\|_{\infty} \eps+ \sqrt{h}C\left(\frac{1}{r_0}+\int_{0}^\infty r dF(r)\right).
\end{align*}
Then for $h>0$ small enough,
$$
\sup_{S \text{ closed}}\left\| \tilde J^h (\cdot, S) - J(\cdot, S)\right\|_{^\infty} \leq (2\|f\|_{\infty}+1) \eps.
$$
The proof is complete.
\end{proof}

\begin{Definition}
Let $\eps\geq 0$. A closed set $S\subset\X$ is said to be an $\eps$-mild equilibrium (in continuous time), if
	\be 	\notag f(x)\leq J(x,S)+\eps,\quad \forall x\notin S.\ee
	Denote $\Ec_\eps$ the set of $\eps$-mild equilibria.
\end{Definition}
{\color{black} We only impose criterion on the part $x\notin S$ in the above definition, since it trivially holds that $f(x)+\eps\geq J(x,S)$ $\forall x\in S$, as explained in Remark \ref{rm:def.equilibrium}.} Define $V_\eps(x):=\sup_{S\in\Ec_\eps}J(x,S)$. The following lemma is from \cite[Proposition 3.5]{bwz-2022}.
\begin{Lemma}\label{lm:eps} 
{\color{black}Let Assumption \ref{a2} hold.} Then
$$
\lim_{\eps\to 0+}V_\eps(x) = V(x),\quad \forall x\in \X.
$$
\end{Lemma}

\begin{proof}[\bf Proof of Theorem \ref{thm:semicts}]
Let $\eps>0$. By Proposition \ref{lm:main} there exists $h_0>0$ such that for any $h\leq h_0$
\be\notag
f(x(h))\leq \tilde J^h(x(h), S_*^h)\leq J(x(h), S_*^h)+\frac{\eps}{2}, \quad \forall x(h)\notin S_*^h.
\ee 
where the first inequality follows from $S_*^h$ being an equilibrium in the discretized model. Then by Corollary \ref{lm:uniflip} and \eqref{eq:S.xunf}, there exists $h_1\in(0,h_0)$ such that for any $h\leq h_1$ no matter $d(x,S_*^h)<3c_2\sqrt{h}$ or not ($c_2$ is from Assumption \ref{a3}),
$$
f(x)\leq J(x, S_*^h)+\eps, \quad \forall x\in \X, 
$$
Thus for $h>0$ small enough, $S_*^h$ is an $\eps$-mild equilibrium in continuous time and $V^h(\cdot)=\tilde J^h(\cdot,S_*^h)\leq V_\eps(\cdot)$. Consequently, 
\begin{equation}\label{eqq01}
\limsup_{h\to 0+}V^h(x)=\lim_{\eps\to 0+}\limsup_{h\to 0+}V^h(x)\leq \lim_{\eps\to 0+}(V_\eps(x)+\eps) = V(x),\quad \forall x\in \X,
\end{equation}
where the last inequality follows from Lemma \ref{lm:eps}. 
\end{proof}

\begin{proof}[\bf Proof of Theorem \ref{thm:epscts}] 
Let $\eps>0$. By Proposition \ref{lm:main} there exists $h_0>0$ such that for $h\leq h_0$,
\begin{equation}\label{e4}
\sup_{S \text{ closed}}\| J(x, S)-\tilde J^h(x,S)\|_{\infty}\leq \eps.
\end{equation}
Then for any $h\leq h_0$ and $y\notin S_*(h)$,
\begin{equation}\label{e5}
f(y)\leq J(y, S_*) \leq \tilde J^h(y, S_*)+\eps=\tilde J^h(y, S_*(h))+\eps.
\end{equation}
On the other hand, for any $h\leq h_0$ and $y\in S_*(h)$, by \eqref{e4},
$$f(y)=J(y,S_*(h))\geq \tilde J^h(y,S_*(h))-\eps.$$
This together with \eqref{e5} implies $S_*(h)$ is an $\eps$-equilibrium for any $h\leq h_0$. As a result, for $x\in\X$,
\be\label{eqq02}
V(x)=J(x,S_*)=\liminf_{h\to 0+}\tilde J^h(x,S_*)=\liminf_{\eps\to 0+} \liminf_{h\to 0+} \tilde J^h(x,S_*(h))\leq\liminf_{\eps\to 0+} \liminf_{h\to 0+} V^h_\eps(x).
\ee

Take $\eps>0$. Thanks to the Lipschitiz continuity $f$, there exists $h_1>0$ such that $|f(x)-f(y)|\leq \eps$ whenever $|x-y|\leq h_1$.
Then for any $h\leq h_0\wedge h_1$ and any $S\in \Ec_{\eps}^h$ we have that
$$
J(x,S)\geq \tilde J^h(x,S)-\eps=\tilde J^h(x(h),S)-\eps\geq f(x(h))-2\eps\geq f(x)-3\eps,\quad \forall x\notin S.
$$
That is, $S\in \Ec_{3\eps}$. As a consequence, $\Ec^h_{\eps}\subset \Ec_{3\eps}$ for any $h\leq h_0\wedge h_1$. Therefore,
\begin{align}
\label{eqq03}&\limsup_{\eps\to 0+}\limsup_{h\to 0+} V^h_{\eps} (x)=\limsup_{\eps\to 0+}\limsup_{h\to 0+}\sup_{S\in\mathcal{E}_\eps^h}\tilde J^h(x,S)\leq \limsup_{\eps\to 0+}\limsup_{h\to 0+}\sup_{S\in\mathcal{E}_\eps^h} (J(x,S)+\eps)\\
\notag &\leq \limsup_{\eps\to 0+}\limsup_{h\to 0+}\sup_{S\in\mathcal{E}_{3\eps}} J(x,S)= \limsup_{\eps\to 0+}V_{3\eps} (x) =V(x),\quad \forall x\in \X,
\end{align}
where the last equality follows from Lemma \ref{lm:eps}. This completes the proof.
\end{proof}

\begin{Remark}
Using almost the same proof, we can show that Theorems \ref{thm:semicts} and \ref{thm:epscts} still hold if $\tilde J^h$ is replaced by $J^h$ in the definition of $V^h$ and $V^h_\eps$ in \eqref{e6}. Indeed, in this case, it is easy to see that \eqref{eqq01} and \eqref{eqq03} still hold as $\tilde J^h(\cdot,S)\geq J^h(\cdot,S)-\eps$ for $S\in\mathcal{E}_\eps^h$. The last inequality in \eqref{eqq02} also holds, since if $x\in S_*$ then $J(x,S_*)=f(x)\leq V_\eps^h(x)+\eps$, and if $x\notin S_*$ then for $h>0$ small enough, $x\notin S_*(h)$ and thus $\tilde J^h(x,S_*(h))=J^h(x,S_*(h))$.
\end{Remark}

{\color{black}
	\begin{Remark}
		As mentioned in Remark \ref{rm:def.equilibrium}, any Borel stopping region $S$ and its closure $\overline S$ share the same value function. That is to say, we can release $S$ to be only Borel instead of being closed in Definition \ref{def.mild.optimal}, and let $\Ec$ denote the set of all Borel mild equilibria. Then Lemma \ref{lm.optimalmild} reads $S_*:= \cap_{S\in \Ec} \overline S$ as the smallest closed mild equilibrium and 
		$$
		V(x)=\sup_{S\in \Ec} J(x, \overline S)= \sup_{S\in \Ec} J(x, S)=\sup_{S\in \Ec \text{ is closed}} J(x, \overline S)=J(x, S_*).
		$$
		Then Theorems \ref{thm:semicts} and \ref{thm:epscts} are still valid.
	\end{Remark}
}

\section{An example of strict upper semi-continuity}\label{sec:4}

{\color{black}In this section, we provide an example showing that it is possible to have the strict inequality for \eqref{eq001}, i.e., $\limsup_{h\to0+} V^h(x)<V(x)$ for some $x\in \X$.} Proposition \ref{p11} is the main result of this section.

Let $X$ be a standard Brownian motion and $\Delta t = h=\frac{1}{3^{2n}}$ for $n\in \N_+$. Then by \eqref{eq:discrete.x} and \eqref{eq:discrete.P}, $x^h_k=\frac{k}{3^{n}}$ and $p^h_{k,+}= p^h_{k,-}=\frac{1}{2}$ for $k\in \Z$. Denote $x^*:=\frac{1}{2}$. Assume $F(\cdot)$ in \eqref{e3} is the uniform distribution on $[1,2]$. Let
\bee 
\begin{aligned}
	f_0(x):=& \E_x\left[ \delta\left( \rho_{\{1/2\}} \right) \right]=\int_1^2 e^{-\sqrt{2r} |x-1/2|}dr,\quad \text{for } x\in \R;\\
	f_1(x):=& \E_x\left[\delta \left( \rho_{\left\{ 0,1/2 \right\} } \right) \cdot 1_{\left\{ \rho_{\left\{0,1/2 \right\}}=\frac{1}{2} \right\}}\right]+f_0(0)\E^x\left[\delta(\rho_{\{0,1/2\}})\cdot 1_{\{\rho_{\{0,1/2\}}=0\}}\right]\\
	=&\int_1^2 \frac{e^{\sqrt{2r} x}- e^{-\sqrt{2r} x}}{e^{\sqrt{2r} x^*}-e^{-\sqrt{2r} x*}}+ f_0(0) \frac{e^{\sqrt{2r} (x^*-x)}- e^{-\sqrt{2r} (x^*-x)}}{e^{\sqrt{2r} x^*}-e^{-\sqrt{2r} x*}}dr, \quad \text{for } x\in [0,x^*].
\end{aligned}
\eee  
Then define the reward function as
\bee 
f(x):=
\begin{cases}
	e^{x}f_0(0)  \E_x\left[ \delta \left( \rho_{\{0\}} \right) \right]& x\in (-\infty, 0],\\
	\frac{1}{1+K_1x\left( \frac{1}{2}-x \right)}f_1(x) & x\in (0,x^*],\\
	e^{K_2(x^*-x)}f_0(x), &  x\in (x^*,\infty),
\end{cases}
\eee
where the constants $K_1,K_2>0$ are chosen such that
\begin{align}
\label{eqq04}  	f'(0+)=& f'_1(0+)-\frac{K_1}{2}f_0(0)\leq -17,\quad f'(x^*-)=f'_1(x^*-)+\frac{K_1}{2}\geq 17,\\
\label{eqq05}	f'(x^*+)=&f_0'(x^*+)-K_2<-f'(x^*-).
\end{align}
{\color{black}Then, $f$ is bounded and Lipschitz continuous over $\X$.}

\begin{Lemma}\label{lm:eg}
	Let $S:=\{0,x^*\}$. With a bit of abuse of notation, denote $J^h(k;r):=J^h(x, S;r)$ for $r\in[1,2]$ and $J^h(k):=\int_1^2J^h(k;r)dr$, for $x=k\sqrt{h}$ with $k\in \Z$. Then for $n$ big enough we have that
	\begin{align}
		J^h(k)\geq f(k\sqrt{h}),\quad\forall\, k\leq \xhstarv \label{eq:eg.lm}.
	\end{align}
	
\end{Lemma}

\begin{proof}
	We first prove \eqref{eq:eg.lm} for $k=0,..., \xhstarh$. Fix $r\in[1,2]$. We have that
	$$
	\begin{cases}
		J^h(k;r)=\frac{(1+rh)^{-1}}{2}(J^h(k-1;r)+J^h(k+1;r)),\quad k=1,...,\xhstarh-1, \\
		J^h(0;r)=f(0), \quad J^h\left( \xhstarv;r \right)=f\left( x^*-\frac{1}{2}\sqrt{h} \right).
	\end{cases}
	$$
	Using the characteristic equation, we get that
	\bee 
	J^h(k;r)=\frac{a_h^k-a_h^{-k}}{a_h^{\xhstarh}-a_h^{-\xhstarh}} f\left( x^*-\frac{1}{2}\sqrt{h} \right)+ \frac{a_h^{\xhstarh-k}-a_h^{-\left( \xhstarh-k \right)}}{a_h^{\xhstarh}-a_h^{-\xhstarh}} f(0),\quad k=0,..., \xhstarv,
	\eee 
	where 
	\be\label{eq:eg.ah}
	a_h=\frac{1+\sqrt{1-(1+rh)^{-2}}}{(1+rh)^{-1}}=1+ r h+\sqrt{r^2 h^2+2 r h}=\left(1+ r h-\sqrt{r^2 h^2+2 r h} \right)^{-1}.
	\ee 
	A direct calculation gives that
	\bee\label{eq:eq.ah'} 
	\lim_{n\to\infty} \frac{1-a_h}{\sqrt{h}}=-\sqrt{2r}, \quad  \lim_{n\to\infty} \frac{1-a_h^{-1}}{\sqrt{h}}=\sqrt{2r}.
	\eee
	Take $N_1\in\N$ such that $\left| \frac{1-a_h}{\sqrt{h}} \right|, \left| \frac{1-a_h^{-1}}{\sqrt{h}} \right| \leq \sqrt{2r}+\frac{1}{3}$ for all $n\geq N_1$. Set $y:=\sqrt{2rh}$ and 
	\be\label{eq:eg.lm5}  
	g(y):=a_h^{\frac{1}{\sqrt{2r h}}}=\left(1+ r h-\sqrt{r^2 h^2+2 r h} \right)^{-\frac{1}{\sqrt{2rh}} }=\left[1-\left(\sqrt{y^2+y^4/4}-y^2/2\right)\right]^{-\frac{1}{y}}.
	\ee
	We can compute that
	$$\ln(g(y))=1-\frac{1}{24}y^2+O(y^3),$$
	which implies 
	$ 
	g(y) \to e \text{ as $n\to\infty $},
	$
	and thus $a_h^{\xhstarh} = e^{(1-\frac{1}{24}y^2+O(y^3))\sqrt{2rh} \xhstarv} \to  e^{\sqrt{2r}x^*}\geq e^{x^*}$ for $r\in[1,2]$. Since $a_h^k+a_h^{-k}\leq a^{\xhstarh}+a^{-\xhstarh}$ for all $k=0,1,..., \xhstarh$ and $\frac{e^{x^*}+e^{-x^*}}{e^{x^*}-e^{-x^*}} \approx 2.163953414$, we can take $N_2\geq N_1$ such that for all $n\geq N_2$, we have
	\be\label{eq:eg.lm7} 
	\frac{a_h^{k}+a_h^{-k}}{a_h^{\xhstarh}-a_h^{-\xhstarh}}\leq 2.5,\quad \forall r\in[1,2],\  k=0,...,\xhstarh. 
	\ee 
	
	Now for $r\in[1,2]$ and $k=1,..., \xhstarh$, we have that 
	\begin{align*}
		& \frac{| J^h(k;r)- J^h(k-1; r)|}{\sqrt{h}} \\
		&=\left|  \frac{a_h^k(1-a_h^{-1})-a_h^{-k}(1-a_h)}{\sqrt{h} \left( a_h^{\xhstarh}-a_h^{-\xhstarh}\right)} f\left( x^*-\frac{1}{2}\sqrt{h} \right)
		+\frac{a_h^{\xhstarh-k}(1-a_h)-a_h^{-\left( \xhstarh-k \right)}(1-a_h^{-1})}{\sqrt{h} \left( a_h^{\xhstarh}-a_h^{-\xhstarh}\right)} f\left( 0\right)   \right|\\
		&\leq \left| \sqrt{2r}+\frac{1}{3} \right| \left|  \frac{a_h^k+a_h^{-k}}{ a_h^{\xhstarh}-a_h^{-\xhstarh}} f\left( x^*-\frac{1}{2}\sqrt{h} \right)
		+ \frac{a_h^{\xhstarh-k} + a_h^{-(\xhstarh-k)}}{ a_h^{\xhstarh}-a_h^{-\xhstarh}} f\left( 0\right)   \right|\\
		&\leq  5(\sqrt{2r}+\frac{1}{3})\quad \forall n\geq N_2,
	\end{align*}
	where the last line follows from \eqref{eq:eg.lm7}.
	Thus, for $k=1,...,\xhstarh$,
	\be\label{eq:eg.lm3}  
	\begin{aligned}
		\frac{| J^h(k)- J^h(k-1)|}{\sqrt{h}} \leq & \int_1^2 \frac{| J^h(k;r)- J^h(k-1; r)|}{\sqrt{h}} dr <15,\quad \forall n\geq N_2.
	\end{aligned}
	\ee 
	By \eqref{eqq04}, there exists $\alpha\in(0,1/4)$ such that 
	$$f'(t+)\leq -16, f'((x^*-t)-)\geq 16 \quad \forall t\in[0,\alpha].
	$$ 
	Combining the above with \eqref{eq:eg.lm3}, we have 
	\be\label{eq:eg.lm1}   
	J^h(k)\geq f(k\sqrt{h}), \text{ if } k\sqrt{h}\in [0,\alpha]\cup[x^*-\alpha, x^*].
	\ee 
	Notice that $J^h(x, S)\to f_1(x)$. By Proposition \ref{lm:main} and the definition of $f$, there exists $N_3\geq N_2$ such that $J^h(\cdot, S)\geq f(\cdot)$ on $[\alpha, x^*-\alpha]$ for all $n\geq N_3$. This together with \eqref{eq:eg.lm1} tells that \eqref{eq:eg.lm} holds for $k=0,1,...,\xhstarh$ whenever $n\geq N_3$.
	
	Next, we verify \eqref{eq:eg.lm} for $k<0$. Take $r\in[1,2]$. By \eqref{eq:eg.ah} and \eqref{eq:eg.lm5}, we have that for $k< 0$
	\bee 
	\begin{aligned}
		J^h(k;r)	= f_0(0) a_h^{k}=f_0(0)\left( a_h^{\frac{1}{\sqrt{2rh}}} \right)^{\sqrt{2r} k\sqrt{h}}= f_0(0)\left( e^{1-\frac{1}{24}y^2 +O(y^3)}\right)^{ \sqrt{2r} (k\sqrt{h})},
	\end{aligned}
	\eee 
	where the first equality above is implied by the characteristic equation. As $r\in[1,2]$, we can find $N>0$ independent of $r$ such that $e^{1-\frac{1}{24}y^2 +O(y^3)}\leq e$ for any $r\in[1,2]$ and $n\ge N$. Hence, $J^h(k;r)\geq f_0(0)e^{\sqrt{2r}k\sqrt{h}}$ and thus 
	$$J^h(k) = \int_1^2 J^h(k;r) dr \geq f(k\sqrt{h}),\quad \text{for } k<0,n\geq N.$$	
\end{proof}

\begin{Proposition}\label{p11}
	$S_*= \{ 1/2\}$ and $V(x)= f_0(x)$ for $x\in \R$. Meanwhile, $0\in S_*^h$ for $n\in \N_+$ large enough and $\limsup_{n\to\infty} V^h(x)< V(x)$ for all $x<0$.
\end{Proposition}

\begin{proof}
	By the definition of $f_0$ and $f$, we have that
	$$f(x)\leq J(x, \{1/2\}),\quad \forall x\in \R,$$
	and $f$ achieves the global maximum at $1/2$. Thus, $S_*=\{1/2\}$ and $V(x)=J(x, \{1/2\})=f_0(x)$. 
	
	Now we prove that for all $n$ big enough,
	\be\label{eq:eg.0}  
	J^h(0, \{1/2\})<J(0, \{1/2\})=f(0).
	\ee 
	Take $r>0$. We have that
	\bee 
	J^h(0, \{1/2\}; r)=f\left(x^*-\frac{1}{2}\sqrt{h}\right) \E_0 \left[ (1+rh)^{\rho_{\{1/2\}}(h)} \right],
	\eee 
	where 
	\be\label{eq:eg.1} 
	f(x^*(h))=f\left(x^*-\frac{1}{2}\sqrt{h}\right)=f(1/2)-\frac{f'_{-}(1/2)}{2} \sqrt{h} +O(h).
	\ee 
	By using characteristic equation and combining with \eqref{eq:eg.ah} and \eqref{eq:eg.lm5} again, we have that
	\bee 
	\begin{aligned}
		\E^h_0\left[ (1+rh)^{-\rho_{\{1/2\}}(h)} \right]=& \left( \frac{1-\sqrt{1-(1+rh)^{-2} }}{(1+ r h)^{-1}} \right)^{\left\lfloor\frac{x^*}{\sqrt{h}} \right\rfloor}=e^{\left(1-\frac{1}{24}y^2+O(y^3) \right)\left(-\sqrt{2r}x^* A_h \right)},
	\end{aligned}
	\eee 
	where $A_h:=\left\lfloor\frac{x^*}{\sqrt{h}} \right\rfloor \big/\frac{x^*}{\sqrt{h}}$.
	By \eqref{eq:eg.1}, $J(0, \{1/2\}; r)= e^{-\sqrt{2r}x^* }f(1/2)$ and the above equality,
	\be\label{eq:eg.3}   
	\begin{aligned}
		&J^h(0, \{1/2\}; r)-J(0, \{1/2\}; r)\\
		&= \left(e^{\left(1-\frac{1}{24}y^2+O(y^3) \right)\left(-\sqrt{2r}x^* A_h \right)}-e^{\left(1-\frac{1}{24}y^2+O(y^3) \right)\left(-\sqrt{2r}x^* \right)}\right)f(1/2)\quad \text{(I)}\\
		&\quad + \left(e^{\left(1-\frac{1}{24}y^2+O(y^3) \right)\left(-\sqrt{2r}x^* \right)}-e^{-\sqrt{2r}x^* }\right)f(1/2) \quad \text{(II)}\\
		&\quad -\frac{f'_{-}(1/2)}{2} \sqrt{h} \cdot e^{\left(1-\frac{1}{24}y^2+O(y^3) \right)\left(-\sqrt{2r}x^* A_h \right)} \quad \text{(III)}+O(h).
	\end{aligned}
	\ee 
	Denote $B:= (1-\frac{1}{24}y^2+O(y^3))\sqrt{2r}x^*$. Notice that $0\leq 1-A_h=\frac{\sqrt{h}}{2x^*}= \sqrt{h}$. We have that 
	\begin{align*}
		\text{(I)}&= e^{-BA_h}(1-e^{-B\sqrt{h}})f(1/2)=e^{-BA_h}\left( \sqrt{2r}x^*\sqrt{h}+O(h) \right)f(1/2),\\
		\text{(II)}&=\left(e^{-\sqrt{2r}x^*+O(h) }-e^{-\sqrt{2r}x^* }\right)f(1/2)=O(h)f(1/2),\\
		\text{(III)}&=-\frac{f'(x^*-)}{2} e^{-BA_h} \sqrt{h}.
	\end{align*}
	By \eqref{eqq04}, we have $\frac{f'(x^*-)}{2}>1\geq\sqrt{2r}x^*$ for any $r\in[1,2]$. This together with \eqref{eq:eg.3} implies that for $n$ big enough,
	$$J^h(0, \{1/2\}; r)-J(0, \{1/2\}; r)<0\quad  \forall r\in[1,2].$$
	As a consequence, \eqref{eq:eg.0} holds.
	
	Since $f(x^*)$ is the maximum value of $f$ and $f'(x^*+)<-f'(x^*-)$ by \eqref{eqq05}, we have $f(x^*-\frac{1}{2}\sqrt{h})>f(x^*+\frac{1}{2}\sqrt{h})$ on $\{k\sqrt{h}:\ k\in \Z\}$ for $n$ large enough. Hence, $x^*(h)\in S^h_*$ for large $n$. By Lemma \ref{lm:eg}, with $S=\{0\}\cup[x^*,\infty)$, $S(h)$  is a pseudo equilibrium (see \cite[Definition 4.2]{bwz-2023}) for the discretized model when $n$ is large enough. Suppose $R\subset\{k\sqrt{h}:\ k\in \Z\}$ is another pseudo equilibrium. Then by \cite[Lemma 4.2(a)]{bwz-2023}, $R\cap S(h)$ is also a pseudo equilibrium. If $0\notin R$, then $R\cap S(h)\cap(-\infty,x^*]=\{x^*(h)\}$. Then \eqref{eq:eg.0} would contradict $R\cap S(h)$ being a pseudo equilibrium. Therefore, $0\in R$. By \cite[Proposition 4.2]{bwz-2023} $S_*^h$ coincides with the smallest pseudo equilibrium. As a result, for $n$ large enough, $S_*^h\cap(-\infty,x^*]=\{0,x^*(h)\}$.
	Then for any $x_0<0$, we have that 
	$$
	\limsup_{n\to\infty}V^h(x_0) =\limsup_{n\to\infty}J^h(x_0, \{0\}) =J(x_0, \{0\})<J(x_0, \{1/2\}),
	$$
	where the inequality follows from the decreasing impatience property \eqref{eq.delta.logsub}. The proof is complete.
\end{proof}

\section{An example on American put option}\label{sec:5}
We consider the stopping problem for an American put option. We construct the optimal equilibria for the continuous-time and binomial-tree model, and show the convergence of value functions. Proposition \ref{proex} is the main result of this section.

Let $\delta(t)=\int_1^2 e^{-rt}dF(r)$ for some distribution function $F$ supported on $[1,2]$, and therefore the assumption $\int_0^\infty rdF(r)<\infty$ is fulfilled. Let $Z$ be a one-dimensional geometric Brownian motion,
$$
dZ_t = \mu Z_t dt + \sigma Z_t dW_t,\quad Z_0=z\in (0,\infty).
$$
The payoff function is given by $g(z):=(K-z)^+$. This problem is on exercising a perpetual American put option with hyperbolic discounting. It can also be viewed as a real option problem, where a company evaluates an investment plan with a constant payoff $K$ and a stochastic cost $Z$, seeking to determine an optimal time for execution. Lemma \ref{lm:egnew.0} presents the optimal equilibrium strategy for making the investment, and Proposition \ref{proex} demonstrates that the company can use the binomial tree model to estimate this equilibrium strategy.


For simplicity, we set $\sigma=1$ and assume $\nu := \mu-1/2\geq -1/2$. Then $Z$ is a submartingale. (At this stage assumptions in our theory are not satisfied for $Z$. We will change $Z$ to $X :=\ln Z$ later so that these assumptions will be satisfied.) We use the notation $\tilde \E$ for the expectation under $Z$. For $z\in(0,\infty)$ and $\tilde a \in (0, K)$, define 
$$\eta(z; \tilde a):=\tilde\E_z\left[ \delta\left( \rho_{\{\tilde a\} }  \right) g\left( Z_{\rho_{\{ \tilde a \}} } \right)   \right].\footnote{Here with a bit of abuse of notation, $\rho$ is understood as a stopping time for $Z$ and $\rho_{\{a\}}$ represents the first time for $Z$ entering the set $\{\tilde a\}$. For the rest of this section, $\rho_S$ is understood as the first time for the underlying process (either discretised or continuous) to enter $S$. }$$

\begin{Lemma}\label{lm:egnew.0}
For any $0<\tilde a<K$, the map $z\mapsto\eta(z; \tilde a)$ on $[\tilde a,\infty)$ is strictly decreasing and strictly convex, and the map $\tilde a\mapsto\eta'(\tilde a+; \tilde a)$ is increasing $(0, K)$. Moreover, $(0,\tilde a]$ is a mild equilibrium if and only if $\tilde a\geq  \frac{\lambda}{1+\lambda}K$ with  
$$
\lambda:= \int_1^2 \left( \sqrt{ \nu^2 +2r }+\nu \right) dr>0.
$$

\end{Lemma}

\begin{proof}
Take $0<\tilde a<K$. By \cite[2.0.1 on page 628]{MR1912205}
$$
\eta(z;\tilde a) = (K-\tilde a )\int_1^2 \tilde \E_z[e^{-r \rho_{\{\tilde a \}}}] dr  = (K-\tilde a)\int_1^2 \left(\frac{\tilde a}{z}\right)^{\sqrt{\nu^2+2r}+\nu} dr,\quad z\geq \tilde a.
$$
Then a direct calculation shows that 
\begin{align*}
\eta'(z+;\tilde a) = & -\frac{K-\tilde a }{ z }\int_1^2 (\sqrt{\nu^2+2r}+\nu) (\frac{\tilde a}{z})^{\sqrt{\nu^2+2r}+\nu} dr <0,\quad  z\geq \tilde a;\\
\eta''(z+;\tilde a) = & \frac{K-\tilde a }{ z^2 }\int_1^2 (\sqrt{\nu^2+2r}+\nu) (\sqrt{\nu^2+2r}+\nu+1)(\frac{\tilde a}{z})^{\sqrt{\nu^2+2r}+\nu} dr >0,\quad  z\geq \tilde a.
\end{align*}
Moreover, $\eta'(\tilde a+; \tilde a) = -\frac{K-\tilde a }{\tilde a}\int_1^2 (\sqrt{\nu^2+2r}+\nu) dr$ increases on $(0,K)$. We can then conclude that $(0,\tilde a]$ is a mild equilibrium if and only if $\eta'(\tilde a+;\tilde a)=-\frac{K-\tilde a}{\tilde a}\lambda\geq -1$.
\end{proof}


To ensure the discretization can be constructed aligning with Assumption \ref{a3}, we set $X_t = \ln(Z_t)$ which satisfies
$$dX_t = \nu dt +dW_t$$
with $X_0=x\in \R$, and the reward function becomes $f(x):=(K-e^x)^+$. Denote $\E$ the expectation under $X$. We have
$$
J(x;a):=\E_x\left[ \delta\left( \rho_{(-\infty, a]} \right) f\left( X_{\rho_{(-\infty, a]} } \right) \right]= \eta(e^x; e^a), \quad  x\in \R, \ a\in (-\infty, \ln K),
$$
Then the following is immediate from Lemma \ref{lm:egnew.0}.
\begin{Corollary}\label{cor:egnew}
 $(a,\infty)$ with $a\in (-\infty, \ln K )$ is a mild equilibrium under $X$ if and only if
 $$a\geq  \ln\left(\frac{\lambda}{1+\lambda}K\right),$$
and $S_*:=(-\infty, \ln(\frac{\lambda}{1+\lambda} K))$ is the unique optimal mild equilibrium under $X$. 
\end{Corollary}

Take $h>0$. The discretization under $X$ is given by $d_{k,\pm}=\sqrt{h}$, $x^h_k=\sqrt{h} k$ and
$$\P(X^h_{t+1} = x^h_k\pm \sqrt{h} \mid X^h_t=x^h_k)=p^h_{k,\pm}=\frac{1 \pm \nu\sqrt{h}}{2},\quad k\in \Z.$$
Equivalently, the discretization $Z^h$ for $Z$ reads $z^h_k=e^{x^h_k}=e^{\sqrt{h}k}$ and $\P(Z^h_{n+1}/Z^h_n)=e^{\pm \sqrt{h}}= p^h_{k,\pm}$. Moreover, $e^{\sqrt{h}} p^h_{k, +} + e^{-\sqrt{h}} p^h_{k, -} >1$, which shows that $Z^h$ is also a submartingale for any $h>0$. We shall use $\E^h$ (resp. $\tilde \E^h$) for the expectation under $X^h$ (resp. $Z^h$).


Define
$$\xi^h:=\inf\{k\in\N:\ X_k^h=0\,|\,X_0^h=\sqrt{h}\}\quad\text{and}\quad a^h:=\frac{1- \E[\delta^h(\xi^h)]}{ e^{\sqrt{h}} - \E[\delta^h(\xi^h)] } K.$$
Recall $\lambda$ defined in Lemma \ref{lm:egnew.0}. We have the following result.

\begin{Lemma}\label{lem001}
We have that
$$\lim_{h\to0+}a^h =  \frac{\lambda K}{1+\lambda}.$$
\end{Lemma}

\begin{proof}
It is equivalent to show that
\be\label{eq:egnew.prop1} 
\lim_{h\to 0+}\frac{1-\E^h[\delta^h(\xi^h)]}{e^{\sqrt{h}}-1} =\lambda.
\ee 
By the generating function, we have that for $x\in \X^h\cap (a,\infty)$
$$
\E^h[(1+rh)^{-\xi^h }] = F((1+rh)^{-1}) \text{ with } F(l) = \frac{1-\sqrt{1-4 p^h_{k,+}p^h_{k,-} l^2 }}{2p^h_{k,+} l }.
$$
That is,
\begin{align}\label{equ001}
\E^h[(1+rh)^{-\xi^h }]=&   \frac{1+rh-\sqrt{r^2 h ^2 + (2r+\nu^2) h}}{1+ \nu \sqrt{h}}.
\end{align}
Then
\begin{align*}
\lim_{h\to 0+}\frac{\E^h[(1+rh)^{-\xi^h }]-1}{e^{\sqrt{h}}-1} &= \lim_{h\to 0+}\frac{\frac{1+rh-\sqrt{r^2 h ^2 + (2r+\nu^2) h}}{1+ \nu \sqrt{h}}-1}{e^{\sqrt{h}}-1}\\
&=  \lim_{h\to 0+}\frac{\frac{1-\sqrt{2r+\nu^2}\sqrt{h}}{1+ \nu \sqrt{h}}-1+o(\sqrt{h})}{\sqrt{h}+o(\sqrt{h})}= -\sqrt{2r+\nu^2} -\nu.
\end{align*}
As a result, 
\begin{align*}
&\lim_{h\to 0+}\frac{1-\E^h[\delta^h(\xi^h)]}{e^{\sqrt{h}}-1} = \lim_{h\to 0+}\frac{1-\int_1^2 \E^h[(1+rh)^{-\xi^h }] dF(r) }{e^{\sqrt{h}}-1} = \lim_{h\to 0+}\int_1^2 \frac{1-\E^h[(1+rh)^{-\xi^h }]}{e^{\sqrt{h}}-1} dF(r) \\
&= \int_1^2 (\sqrt{\nu^2 +2r} +\nu) dF(r) =\lambda.
\end{align*}
\end{proof}

\begin{Lemma}\label{lm:egnew.1}
Let $h>0$ be small enough such that $a_h<K-2\sqrt{h}$ (thanks to Lemma \ref{lem001}). Define
$$x^h:=\inf\{ x\in \X^h: x\geq \ln(a^h)\}.$$
Then $S^h_* = (-\infty, x^h]\cap \X^h$ (under $X^h$).
\end{Lemma}

\begin{proof}
It is equivalent to show that $(0, e^{x^h}]\cap \tilde \X^h$ is the optimal equilibrium under $Z^h$. By \cite[Proposition 4.2 ]{bwz-2023}, it suffices to show that $(0, e^{x^h}]\cap \tilde \X^h$ is the smallest pseudo equilibrium (see \cite[Definition 4.2]{bwz-2023} for the definition of pseudo equilibria). 

\textbf{Step 1.} Let $T$ be a pseudo equilibrium under $Z^h$ such that $T\cap[K,\infty)=\emptyset$. We show that $T$ has the form $T=(0,\tilde a_T]\cap \tilde \X^h$ with $\tilde a_T:=\sup \{ y\in T\}< K$. Indeed, suppose this does not hold. Take $z\in \tilde \X^h$ with $z\notin T$ and $z<\tilde a_T<K$, and set $l:= \max\{y\in \tilde X^h: y< z, y\in T \}$ (here $\max\emptyset:=0$) and $r:= \min\{y\in \tilde X^h: y> z, y\in T \}<K$. $Z^h$ being a submartingale tells that
$$g(z)=K-z>\tilde \E^h_z[\delta(\rho_T) g(Z^h_{\rho_T})].$$
This would contradicts $T$ being a pseudo equilibrium.

\textbf{Step 2.} We show that if $T=(0,y]\cap \tilde \X^h$ for some $y\in\tilde\X\cap(0,K)$ is a pseudo equilibrium under $Z^h$, then $y\geq a^h$. Notice that $\frac{1- \E^h[\delta^h(\xi^h)]}{ e^{\sqrt{h}} - \E^h[\delta^h(\xi^h)] }  \leq  e^{-\sqrt{h}}$, so the statement holds trivially for $y \geq  K e^{-\sqrt{h}}$. Now we assume $y<Ke^{-\sqrt{h}}$. Since $e^{\sqrt{h}}>1$,we have $y<ye^{\sqrt{h}}<K$, which implies that $ye^{\sqrt{h}}\notin T$. Then
\be\notag
K-ye^{\sqrt{h}}=g(ye^{\sqrt{h}})\leq \tilde \E^h_{ye^{\sqrt{h}}}[\delta^h(\rho_T) g(\rho_T)]=(K-y) \tilde \E^h_{ye^{\sqrt{h}}}[\delta^h(\rho_T) ],
\ee
This implies $y \geq a^h$.

\textbf{Step 3.} Set $y^h:= e^{x^h} = \min \{ y\in \tilde \X, y\geq a_h\}$ and $T_h:=(0,y^h]\cap \tilde \X^h$. We show that $T_h$ is a pseudo equilibrium under $Z^h$. 
Take $z=y^h e^{k\sqrt{h}}\in\X^h\setminus T^h$ for some $k\in \N_+$. As $e^{\sqrt{h}}>1$ and $y^h\geq a^h= K\frac{1- \E^h[\delta^h(\xi^h)]}{ e^{\sqrt{h}} - \E^h[\delta^h(\xi^h)] }$, we have that
\be\label{eq:egnew.lm1} 
\begin{aligned}
y^h\geq & K\frac{1- \E[\delta^h(\xi^h)]}{ e^{\sqrt{h}} - \E[\delta^h(\xi^h)] } \cdot \frac{ 1+  \E^h[\delta^h(\xi^h)]+\cdots + ( \E^h[\delta^h(\xi^h)])^{k-1}  }{ e^{(k-1)\sqrt{h}} +  e^{(k-2)\sqrt{h}}  \E^h[\delta^h(\xi^h)]+ \cdots +  (\E^h[\delta^h(\xi^h)])^{k-1}} \\
= & \frac{1- ( \E^h[\delta^h(\xi^h)])^k}{e^{k\sqrt{h}}-( \E^h[\delta^h(\xi^h)])^k} K.
\end{aligned}
\ee 
Notice that 
\be\label{eq:egnew.lm3} 
\begin{aligned}
\tilde \E^h_{z}[\delta^h(\rho_{T_h})]=&\E^h_{k\sqrt{h}}[\delta^h(\rho_{\{0\}})]
\geq \E^h_{\sqrt{h}}[\delta^h(\rho_{\{0\}})]\E^h_{{k-1}\sqrt{h}}[\delta^h(\rho_{\{0\}})]\\
\geq &\cdots \geq  (\E^h_{\sqrt{h}}[\delta^h(\rho_{\{0\}})])^k = (\E^h[\delta^h(\xi^h)])^k,
\end{aligned}
\ee
where the inequalities follows from that $\delta^h(m+n)\geq \delta^h(m)\delta^h(n)$.
Then 
\begin{align*}  
\tilde \E^h_z[\delta^h(\rho_{T_h}) g(Z^h_{\delta^h(\rho_{T_h})})] =& (K-y^h) \tilde \E^h_{z}[\delta^h(\rho_{T_h})]\geq (K-y^h) (\E^h[\delta^h(\xi^h)])^k\\
\geq & \max\{K- y^h e^{k\sqrt{h}},0\} = (K-z)^+ = g(z),
\end{align*}
where the first inequality follows from \eqref{eq:egnew.lm3} and the second inequality follows from \eqref{eq:egnew.lm1}. Hence, $T^h$ is a pseudo equilibrium.
%

\textbf{Step 4.} Now take any pseudo equilibrium $R$ under $Z^h$. By \cite[Lemma 4.2]{bwz-2023}, $T^h\cap R$ is also a pseudo equilibrium. By Steps 1 and 2, $T^h\cap R\supseteq T^h$. This implies $T^h\subset R$ and thus $T^h$ is the smallest pseudo equilibrium. This completes the proof.
\end{proof}

Now we are ready to prove the main result in this section.
\begin{Proposition}\label{proex}
	We have 
	\be\notag
	\lim_{h\to 0+} |V^h(x)-V(x)|=0\quad \forall x\in \X.
	\ee 
\end{Proposition}

\begin{proof}
The result follows from Lemmas \ref{lm:egnew.0}-\ref{lm:egnew.1}, Corollary \ref{lm:uniflip} and Proposition \ref{lm:main}.
\end{proof}

\bibliographystyle{plain}
\bibliography{reference}

\end{document}